\documentclass[11pt,a4paper]{amsart}
\usepackage{amsmath,amssymb, amsbsy}
\usepackage{latexsym}
\usepackage{color}
\usepackage{graphicx}
\usepackage{bm}
\usepackage{mathtools}

\usepackage[english]{babel}
\usepackage{amscd}
\usepackage{mathrsfs}
\usepackage{amsfonts}
\usepackage{amsthm}
\usepackage{xcolor}

\newtheorem{theorem}{Theorem}[section]
\newtheorem{proposition}[theorem]{Proposition}

\newcommand{\R}{\mathbb{R}}
\newcommand{\N}{\mathbb{N}}

\newcommand{\Ha}{\mathcal H_\alpha^{N-1}}

\newcommand{\SA}{{\mathbb S}^{N-1}_\alpha}

\allowdisplaybreaks
\begin{document}

\title[Grushin  spherical  harmonics]{Grushin spherical  harmonics in  non-degenerate\\ dimension one}
 
\author[R. Bramati]{Roberto Bramati}
\address{Dipartimento di Ingegneria Gestionale, dell'Informazione e della Produzione,
Universit\`a degli Studi di Bergamo, Viale Marconi 5, 24044 Dalmine (BG), Italy}
\email{roberto.bramati@unibg.it}

\author[P. Luzzini]{Paolo Luzzini}
\address{Dipartimento di Scienze e Innovazione Tecnologica,
Universit\`a degli Studi del Piemonte Orientale ``Amedeo Avogadro'',
Viale Teresa Michel 11, 15121 Alessandria, Italy}
\email{paolo.luzzini@uniupo.it}

\begin{abstract}
The Grushin Laplacian $\Delta_\alpha = \Delta_x +|x|^{2\alpha}\Delta_y $ is a degenerate elliptic operator in   $\mathbb{R}^{h+k}$ that degenerates on
$\{0\} \times \mathbb{R}^k$. We explicitly compute the Grushin spherical harmonics on the gauge sphere  
$\{(x,y) \in \mathbb{R}^{h+k} : |x|^{2(\alpha+1)}+(\alpha+1)^2|y|^2=1\}$ and the corresponding spectrum  when the non-degenerate dimension $h=1$, a case not covered by the existing literature. The eigenfunctions consist of homogeneous Grushin-harmonic polynomials of degrees $n \equiv 0,1 \pmod{\alpha+1}$, even and odd in the $x$-variable respectively, which can be described in terms of Jacobi polynomials. 
\end{abstract}



\maketitle

\section{Introduction}

Let \(h,k,\alpha\in\mathbb N\),   and set \(N=h+k\). We consider the  following operator in $\mathbb{R}^N$:
\begin{equation*}
\Delta_\alpha := \Delta_x+|x|^{2\alpha}\Delta_y,
\end{equation*}
where \(x\in\mathbb R^h\), \(y\in\mathbb R^k\), and \(\Delta_x\), \(\Delta_y\) denote the classical Laplace operators with respect to the \(x\)- and \(y\)-variables, respectively. Operators of this type are commonly referred to as Grushin or Baouendi--Grushin operators, after the foundational works of Baouendi \cite{Ba} and Grushin \cite{Gru70,Gru71}.

The operator \(\Delta_\alpha\) is not uniformly elliptic. Indeed, its coefficient matrix degenerates on the \(k\)-dimensional subspace
\begin{equation}\label{eq:def-sigma}
\Sigma:=\{(0,y)\in\mathbb R^h\times\mathbb R^k\}
\subset\mathbb R^N.
\end{equation}
Nevertheless, when \(\alpha\) is a positive integer, \(\Delta_\alpha\) can be written as a finite sum of squares of smooth vector fields. More precisely, for \(i=1,\ldots,h\), \(j=1,\ldots,k\), and \(\beta\) a multi-index with  \(|\beta|=\alpha\), let
\begin{equation}\label{vect}
X_i:=\partial_{x_i},
\qquad
Y_{\beta,j}:=
\sqrt{\frac{\alpha!}{\beta!}}x^\beta\partial_{y_j}.
\end{equation}
The multinomial formula then gives
\begin{equation*}
\Delta_\alpha
=\sum_{i=1}^h X_i^2
+\sum_{j=1}^k\sum_{|\beta|=\alpha}
Y_{\beta,j}^2.
\end{equation*}
Moreover, the vector fields in \eqref{vect} satisfy Hörmander's bracket-generating condition: their iterated commutators of length at most \(\alpha+1\) span the tangent space at every point of \(\mathbb R^N\). A well-known result by Hörmander \cite{Hor1967} therefore implies that \(\Delta_\alpha\) is hypoelliptic.

It is also possible to define and consider the Grushin operator when $\alpha >0$ is real, but H\"ormander's result does not directly apply, because the generating vector fields are not smooth, in general. This is one reason why the integer case is often treated separately.
There is also an extensive theory for operators generated by non-smooth vector fields; without
attempting to be exhaustive, we mention the seminal work of Franchi and Lanconelli \cite{FrLa82}.  

  We now turn to the setting considered in the present paper and from now on we assume that the non-degenerate variable is one-dimensional, namely $h=1$. In this case, the Grushin operator reads 
\[
\Delta_\alpha = \frac{\partial^2}{\partial x^2}+|x|^{2\alpha}\Delta_y.
\]

As is already known, see, e.g. De Bie and Lian \cite{DeLi25}, in suitable polar coordinates adapted to the natural anisotropic dilations associated with \(\Delta_\alpha\), the operator decomposes into a radial part and an angular operator \(\mathcal L\) acting on the unit gauge sphere (see Section \ref{prel} for precise definitions). The purpose of the present paper is to explicitly compute the eigenfunctions, i.e. the Grushin spherical harmonics, and the eigenvalues of the angular operator \(-\mathcal L\).

For \(h\geq 2\), explicit bases of Grushin spherical harmonics were obtained by Garofalo and Shen~\cite{GaSh94} in the case \(\alpha=1\), \(k=1\), and more recently by De Bie and Lian~\cite{DeLi25} for arbitrary \(k\geq1\) and integer \(\alpha\geq1\)  (see also Liu~\cite{Li18} for related results on homogeneous polynomial solutions). The case \(h=1\), however, is not covered by those constructions and was explicitly left open in~\cite{DeLi25}, where the authors pointed out that it must be treated separately. Indeed, for $h \geq 2$, the construction separates the angular dependence in
the $x$-variables by expanding in spherical harmonics on $\mathbb{S}^{h-1}$. When $h = 1$, this sphere
reduces to the two points $\mathbb{S}^0 = \{-1, 1\}$, so the relevant distinction is simply a symmetry. For this reason, when we compute the Grushin spherical harmonics,  we obtain  two families, one even  and the other odd in $x$. For $k=1$ these two families are  made of Gegenbauer polynomials, while for $k\geq2$ they are made of Jacobi polynomials.  Moreover, the admissible homogeneous degrees are
$$
n \equiv 0,1  \pmod{\alpha+1},
$$
with corresponding eigenvalue
$$
\frac{n(n+Q-2)}{(\alpha+1)^2},
$$
where $Q:=1+(\alpha+1)k$ is the so-called homogeneous dimension. 

The explicit spectral resolution of $-\mathcal{L}$ is relevant to unique continuation problems, so that addressing this remaining case is of interest. In \cite{GaSh94}, and later in~\cite{DeLi25}, weighted estimates for the Grushin-harmonic spectral projectors, obtained using the explicit form of the Grushin spherical harmonics, are a fundamental ingredient in the derivation of Carleman estimates and strong unique continuation results. In this sense, the construction developed here provides the spectral starting point for investigating analogous projector estimates when \(h=1\). Our results also provide a more explicit description of the asymptotic analysis carried out by Abatangelo, Ferrero and Luzzini in \cite{AbFeLu25}, where the blow-up homogeneity and the angular profile of a solution at a degenerate point are determined by an eigenvalue and an eigenfunction of \(-\mathcal L\). Knowing the spectrum therefore identifies the possible  vanishing orders and the corresponding blow-up profiles explicitly. This information is also potentially useful in singular perturbation problems, where the vanishing order of an eigenfunction determines the leading asymptotic behavior of the perturbed eigenvalues (see \cite{AbLeMu22}).

The paper is organized as follows. In Section~2 we recall the Grushin polar coordinates and the  spectral problem for the angular operator. In Section~3 and  Section~4 we respectively consider the cases \(k=1\) and   \(k\geq2\):  we construct the two families of spherical harmonics, we prove their completeness, and we compute the spectrum and its multiplicities. 

\section{Preliminaries}\label{prel}

In this section we collect some preliminaries about the problem.

\subsection{The Grushin operator in polar coordinates}
We first recall the decomposition of 
\[
\Delta_\alpha = \frac{\partial^2}{\partial x^2}+|x|^{2\alpha}\Delta_y
\]
in polar coordinates adapted to the Grushin geometry. To do this, we need some notation and definitions.
Recall that    $\Delta_y$  in standard polar coordinates in   $\mathbb{R}^k$ can be written as
\[
 \Delta_y  = \frac{\partial^2}{\partial t^2} +\frac{k-1}{t}\frac{\partial}{\partial t} + \frac{1}{t^2}\Delta_{\mathbb{S}^{k-1}} , 
\]
where $t = |y|$ is the radial variable in $\mathbb{R}^k$ and  $\Delta_{\mathbb{S}^{k-1}}$ is the Laplace-Beltrami operator on
 the $(k-1)$-sphere $\mathbb{S}^{k-1}$. Note that if  $k=1$ there are no corresponding angular variables, so that the above formula is understood with $t=y$ and 
 \[
 \Delta_y  = \frac{\partial^2}{\partial t^2}.
\]
In this way \begin{equation}\label{GP}
    \Delta_\alpha  =   \frac{\partial^2}{\partial x^2}  
+|x|^{2\alpha}\left(  \frac{\partial^2}{\partial t^2} +\frac{k-1}{t}\frac{\partial}{\partial t} + \frac{1}{t^2}\Delta_{\mathbb{S}^{k-1}} \right).
 \end{equation}
 Next, we  plan to use as a global radial variable the natural norm associated with the Grushin operator, usually referred to as {\it gauge norm}, given by
 \[
 \rho(x,y) :=\left(|x|^{2(\alpha+1)}+(\alpha+1)^2 |y|^2\right)^{\frac{1}{2(\alpha+1)}}
   \qquad  \forall (x,y)\in  \R^{1+k} \, .
 \]
 To do that, the right change of variables is
 \begin{equation}\label{chvar}
 (x,t) = \Phi(\rho,\varphi) := \left(\rho\, \mathrm{sgn}(\sin \varphi) |\sin \varphi|^\frac{1}{\alpha+1}, \frac{\rho^{\alpha+1}}{\alpha+1}\cos \varphi\right),
 \end{equation}
 with $\rho \in (0,+\infty)$ and 
 \begin{equation}\label{phidef}
 \begin{cases}
 \varphi \in (0,2\pi) \quad &\mbox{ if }\, h=1,k=1\\
 \varphi \in (-\frac\pi2,\frac\pi2) \quad &\mbox{ if }\, h=1,k\geq2.
 \end{cases}
 \end{equation}
 These coordinates were first introduced by Greiner \cite{Gr80} in the Heisenberg group and later adapted to the Grushin operator  by Garofalo and Shen \cite{GaSh94}. The coordinates in the general case can be found in \cite[Formula 5.21]{DaLu03}.
 
 The volume element under the above change of variable has the form
 \begin{equation}\label{vol}
  dxdy=\frac{\rho^{Q-1}}{(\alpha+1)^k} \, |\sin \varphi|^{-\frac{\alpha}{\alpha+1}} |\cos \varphi|^{k-1}
    d\rho \, d\varphi \,   d\mathcal H^{k-1} =\rho^{Q-1}d\rho\,d\mathcal{H}^{N-1}_{\alpha} , 
 \end{equation}
 where   
 \[
d\mathcal{H}^{N-1}_{\alpha}:=\frac{1}{|\nabla \rho|}d\mathcal{H}^{N-1}
 \]
 and $\mathcal H^{k-1}$ and $\mathcal H^{N-1}$ are respectively the Hausdorff measures on the sphere $\mathbb{S}^{k-1}$ and on  
 \[
 \SA:=\left\{(x,y) \in \mathbb{R}^{1+k}: \rho(x,y)=1\right\},
 \]
known as the gauge sphere (see \cite{AbFeLu25} for a detailed justification). We also need to introduce the angular function $\psi_\alpha$, defined by 
 \[
\psi_\alpha(x,y) := \frac{|x|^{2\alpha}}{\rho^{2\alpha}} = |\sin\varphi|^\frac{2\alpha}{\alpha+1} \quad \forall (x,y)\in \mathbb{R}^N \setminus \{0\} .
\]
Then, the Grushin operator in   polar coordinates  can be written as
\begin{equation}\label{GrPo}
\Delta_\alpha   = \psi_\alpha\left(\frac{\partial^2  }{\partial \rho^2}+\frac{(Q-1)}{\rho}\frac{\partial  }{\partial \rho}+\frac{(\alpha+1)^2}{\rho^2}\mathcal{L}\right),
\end{equation}
where $\mathcal{L}$ is the angular operator defined by 
\[
\mathcal{L} := \frac{\partial^2 }{\partial\varphi^2}+ \left(\frac{\alpha}{\alpha+1}\frac{\cos\varphi}{\sin\varphi}  -(k-1)\frac{\sin\varphi}{\cos\varphi}\right)\frac{\partial  }{\partial \varphi}+
\frac{1}{(\cos\varphi)^2}\Delta_{\mathbb{S}^{k-1}}
\]
and 
\[
Q:=  1+(\alpha+1)k
\] is the so-called {\it homogeneous dimension}.
Note that  Garofalo and Shen in \cite[p.135]{GaSh94}  derived formula \eqref{GrPo} in the case $k=1, \alpha=1$. The general case in arbitrary dimensions can be found in De Bie and Lian \cite{DeLi25}.

\subsection{The spectral problem for $-\mathcal{L}$}
We set $L^2_{\alpha}(\SA)$ to be the Hilbert space of square integrable functions on the gauge sphere $\SA$ with respect to the measure $\psi_\alpha d\mathcal{H}^{N-1}_\alpha$, endowed with the scalar product
\[
\langle u,v\rangle_{L^2_{\alpha}(\SA)}:= \int_{\SA} \psi_\alpha u \, v \, d\Ha \qquad
   \forall u,v\in L^2_{\alpha}(\SA) \, .
\]
The spectral problem for the operator $-\mathcal{L}$ can be rigorously defined variationally as 
\[
\int_{\SA} \psi_\alpha  \left[\frac{\partial u}{\partial \varphi}
   \frac{\partial v}{\partial \varphi}+\frac{1}{(\cos \varphi)^2} \, \nabla_{\mathbb{S}^{k-1}} u \cdot \nabla_{\mathbb{S}^{k-1}} v\right] d\Ha = \mu \int_{\SA} \psi_\alpha u \, v \, d\Ha
\]
in a suitable weighted energy space $W^{1,2}_\alpha(\SA)$.  For a detailed treatment and complete definitions we refer to \cite[\S 3]{AbFeLu25}.
 In particular we have the following result (see \cite[\S 3.5]{AbFeLu25}).
\begin{proposition}
The spectrum of the operator $-\mathcal{L}$  is made of a non-negative sequence of eigenvalues $\{\mu_j\}_{j\in\mathbb{N}_0}$ of finite multiplicity, diverging to $+\infty$:
\[
0=\mu_0<\mu_1\leq \mu_2 \leq \cdots \leq \mu_j \to +\infty.
\]
Moreover, its eigenfunctions are of class $C^\infty(\SA)$ and they can be chosen in order to form an orthonormal basis for  $L^2_\alpha(\SA)$.
\end{proposition}

Now, by \eqref{GrPo}, a smooth $\rho$-homogeneous function 
\[
u(\rho, \varphi,\eta) :=  \rho^ng(\varphi,\eta), \quad n \in \mathbb{N}_0
\] 
with $\rho>0$, $\varphi$ as in \eqref{phidef}, $\eta \in \mathbb{S}^{k-1}$ and with the usual convention that $\eta$ must be omitted when $k=1$,
is Grushin-harmonic, i.e., solves  
\[
\Delta_\alpha u= 0,
\]
if and only if
\[
n(n-1)\rho^{n-2}g+(Q-1)n\rho^{n-2}g+(\alpha+1)^2\rho^{n-2}\mathcal{L} g=0.
\] 
This implies 
\begin{equation}\label{eq:eiga}
-\mathcal{L} g  = \frac{n(Q+n-2)}{(\alpha+1)^2}g.
\end{equation}
The homogeneous ansatz then suggests that the eigenvalues of the spherical operator 
$-\mathcal{L}$ have the form 
\[
\frac{n(Q+n-2)}{(\alpha+1)^2}
\]
for some $n \in \mathbb{N}_0$. Our aim is then to find for which $n$ such a number is indeed an eigenvalue and to compute the corresponding eigenfunctions.

\section{The case $k=1$}

Since the  spherical operator and the polar coordinates have different forms for $k=1$ and $k \geq 2$, we separate these two cases.
We first consider   $k=1$. We recall that  in this case the polar coordinates are $\rho \in (0,+\infty) ,\varphi \in (0,2\pi)$ and for a sufficiently smooth function $u$ the Grushin operator  is 
\[
\Delta_\alpha u = \psi_\alpha\left(\frac{\partial^2 u }{\partial \rho^2}+\frac{\alpha+1}{\rho}\frac{\partial u }{\partial \rho}+\frac{(\alpha+1)^2}{\rho^2}\mathcal{L}u \right)\,,
\]
where 
\[
\mathcal{L} =  \frac{\partial^2 }{\partial\varphi^2}+ \frac{\alpha}{\alpha+1}\frac{\cos\varphi}{\sin\varphi}\frac{\partial  }{\partial \varphi}
 .
\]
Then \eqref{eq:eiga} reduces to
\begin{equation}\label{eq:eiga1}
-\mathcal{L} g  = \frac{n(n+\alpha)}{(\alpha+1)^2}g,
\end{equation}
that is 
\begin{equation}\label{eqG1alfa}
 \frac{\partial^2 g}{\partial\varphi^2}+ \frac{\alpha}{\alpha+1}\frac{\cos\varphi}{\sin\varphi}\frac{\partial  g}{\partial \varphi}+\frac{n(n+\alpha)}{(\alpha+1)^2}g=0.
\end{equation}
We now try to find, if possible, solutions to \eqref{eqG1alfa}, following the approach by Garofalo and Shen \cite{GaSh94}. With the  change of  variable $\tau = \cos\varphi$, $h(\tau) = g(\varphi)$ one gets
\[
\frac{\partial g}{\partial \varphi} =-\sin\varphi \frac{\partial h}{\partial \tau}, \qquad 
\frac{\partial^2 g}{\partial \varphi^2}  = \sin^2\varphi \frac{\partial^2 h}{\partial \tau^2}-\cos\varphi\frac{\partial h}{\partial \tau}.
\]
Hence, easy computations show that equation \eqref{eqG1alfa} can be rewritten as
\[
(1-\tau^2)\frac{\partial^2 h}{\partial \tau^2}-\tau\left(\frac{\alpha}{\alpha+1}+1\right) \frac{\partial h}{\partial \tau}+
\left(\frac{n}{\alpha+1}\right)\left(\frac{n}{\alpha+1}+\frac{\alpha}{\alpha+1}\right)h=0.
\]
The above equation is a Jacobi equation of the form
\[
(1-x^2)y''-x(2\lambda+1)y'+N(N+2\lambda)y=0, \qquad \lambda>-\frac12,\, N\in\mathbb{N},
\]
whenever $\frac{n}{\alpha+1} \in \mathbb{N}$, that is, $n \equiv 0 \pmod{\alpha+1}$. In this case, a solution is given by the Gegenbauer polynomial $
C^{\frac{\alpha}{2(\alpha+1)}}_{\frac{n}{\alpha+1}}(\tau)$ (for a reference, see, e.g. \cite[p. 175, eq. (20)]{Er81}). Consequently, a solution of \eqref{eqG1alfa} is given by 
\[
C^{\frac{\alpha}{2(\alpha+1)}}_{\frac{n}{\alpha+1}}(\cos\varphi).
\]
We now also want to construct a solution for the case $n \equiv 1\pmod{\alpha+1}$. In order to do so, we look for a solution to \eqref{eqG1alfa}
in the form
\[
g(\varphi) = \sin\varphi\left|\sin\varphi\right|^{-\frac{\alpha}{\alpha+1}}f(\varphi).
\]
We do so for the following reason: if we want a polynomial solution that is homogeneous of degree $n\equiv 1\pmod{\alpha+1}$, with respect to the natural dilation $\delta_\lambda(x,y) := (\lambda x , \lambda^{\alpha+1}y)$, such a polynomial must surely contain $x$ as a factor, i.e., it must be of the form
$
xP(x,y),
$
with $P(x,y)$ a  polynomial  that is $\delta_\lambda$-homogeneous of degree $n-1$. Note that $x = \rho \sin\varphi\left|\sin\varphi\right|^{-\frac{\alpha}{\alpha+1}}$.

After some computations, we find that $f$ solves the equation 
\[
(1-\tau^2)\frac{\partial^2 f}{\partial \tau^2}-\tau\left(\frac{\alpha+2}{\alpha+1}+1\right) \frac{\partial f}{\partial \tau}+
\left(\frac{n-1}{\alpha+1}\right)\left(\frac{n-1}{\alpha+1}+\frac{\alpha+2}{\alpha+1}\right)f=0.
\]
The above equation is again a Jacobi equation whenever $n \equiv 1 \pmod{\alpha+1}$, and in this case a  solution  is given by
\[
C^{\frac{\alpha+2}{2(\alpha+1)}}_{\frac{n-1}{\alpha+1}}(\tau).
\]
Therefore, a solution of \eqref{eqG1alfa} when  $n \equiv 1 \pmod{\alpha+1}$ is given by 
\[
\sin\varphi\left|\sin\varphi\right|^{-\frac{\alpha}{\alpha+1}}C^{\frac{\alpha+2}{2(\alpha+1)}}_{\frac{n-1}{\alpha+1}}(\cos\varphi).
\]
We have then constructed a family of solutions to \eqref{eqG1alfa} for $n \equiv 0,1\pmod{\alpha+1}$. We set 
\begin{equation}\label{autosp_1}
\mathscr{H}_n := \begin{cases}
\mathrm{span} \left\{C^{\frac{\alpha}{2(\alpha+1)}}_{\frac{n}{\alpha+1}}(\cos\varphi) \right\} \quad &\mbox{ if } n \equiv 0\pmod{\alpha+1}\\
\mathrm{span} \left\{ \sin\varphi\left|\sin\varphi\right|^{-\frac{\alpha}{\alpha+1}}C^{\frac{\alpha+2}{2(\alpha+1)}}_{\frac{n-1}{\alpha+1}}(\cos\varphi) \right\} \quad &\mbox{ if } n \equiv 1 \pmod{\alpha+1},
\end{cases}
\end{equation}

Note that in this case, under the parametrization in the variable $\varphi$ of $\SA$, we have the identification
\[
L_\alpha^2(\SA) = L^2\left((0,2\pi),\frac{1}{\alpha+1}|\sin \varphi|^{\frac{\alpha}{\alpha +1}}  \, d\varphi\right).
\]
Indeed, it is easily seen that on $\SA$
\[
 d\mathcal{H}_\alpha^{N-1} = \frac{1}{\alpha+1}|\sin \varphi|^{-\frac{\alpha}{\alpha+1}}
  \, d\varphi,\qquad  \psi_\alpha = |\sin \varphi|^{\frac{2\alpha}{\alpha+1}}.
 \]
We have the following theorem.
\begin{theorem}\label{ortk1}
Let $k=1$ and let $\mathscr{H}_n$ be as in \eqref{autosp_1}. The following orthogonal decomposition holds:
\[
L^2\left((0,2\pi),\frac{1}{\alpha+1}|\sin \varphi|^{\frac{\alpha}{\alpha +1}} 
  \, d\varphi\right) = \bigoplus_{\substack{n\in \mathbb{N}_0 \\ n \equiv 0,1 \!\!\!\!\pmod{\alpha+1}}}\mathscr{H}_n.
\]
\end{theorem}
\begin{proof}
We first observe that the spaces $\mathscr{H}_n $, where $n\equiv 0,1 \pmod{\alpha+1}$, are mutually orthogonal. Indeed, we recall that for $\lambda > -\frac12 $ the Gegenbauer polynomials $C^\lambda_m(x)$
are a complete set of orthogonal polynomials in $L^2\left((-1,1),(1-x^2)^{\lambda-\frac12} \right)$. If $n_1 \equiv 0$, $n_2 \equiv 0\pmod{\alpha+1}$, $n_1 \neq n_2$, the spaces $\mathscr{H}_{n_1} $ and $\mathscr{H}_{n_2}$ are orthogonal, since
\begin{align*}
\int_{0}^{2\pi}&C^{\frac{\alpha}{2(\alpha+1)}}_{\frac{n_1}{\alpha+1}}(\cos\varphi) C^{\frac{\alpha}{2(\alpha+1)}}_{\frac{n_2}{\alpha+1}}(\cos\varphi)  |\sin \varphi|^{\frac{\alpha}{\alpha +1}}d\varphi\\
& =2 \int_{0}^{\pi}C^{\frac{\alpha}{2(\alpha+1)}}_{\frac{n_1}{\alpha+1}}(\cos\varphi) C^{\frac{\alpha}{2(\alpha+1)}}_{\frac{n_2}{\alpha+1}}(\cos\varphi)  |\sin \varphi|^{\frac{\alpha}{\alpha +1}}d\varphi  \\
&= 2 \int_{-1}^{1}C^{\frac{\alpha}{2(\alpha+1)}}_{\frac{n_1}{\alpha+1}}(\tau)  C^{\frac{\alpha}{2(\alpha+1)}}_{\frac{n_2}{\alpha+1}}(\tau)  (1-\tau^2)^{-\frac{1}{2(\alpha+1)}}d\tau =0.
\end{align*}

 If $n_1 \equiv  1$, $n_2 \equiv 1\pmod{\alpha+1}$, $n_1 \neq n_2$, the spaces $\mathscr{H}_{n_1} $ and $\mathscr{H}_{n_2}$ are orthogonal, since
\begin{align*}
\int_{0}^{2\pi}&C^{\frac{\alpha+2}{2(\alpha+1)}}_{\frac{n_1-1}{\alpha+1}}(\cos\varphi) C^{\frac{\alpha+2}{2(\alpha+1)}}_{\frac{n_2-1}{\alpha+1}}(\cos\varphi) |\sin \varphi|^{\frac{2+\alpha}{\alpha+1}}  d\varphi\\
& =2 \int_{0}^{\pi}C^{\frac{\alpha+2}{2(\alpha+1)}}_{\frac{n_1-1}{\alpha+1}}(\cos\varphi) C^{\frac{\alpha+2}{2(\alpha+1)}}_{\frac{n_2-1}{\alpha+1}}(\cos\varphi) |\sin \varphi|^{\frac{2+\alpha}{\alpha+1}} d\varphi  \\
&= 2 \int_{-1}^{1}C^{\frac{\alpha+2}{2(\alpha+1)}}_{\frac{n_1-1}{\alpha+1}}(\tau)  C^{\frac{\alpha+2}{2(\alpha+1)}}_{\frac{n_2-1}{\alpha+1}}(\tau)  (1-\tau^2)^{\frac1{2(\alpha+1)}}d\tau =0.
\end{align*}

Finally, if $n_1 \equiv  0$, $n_2 \equiv1 \pmod{\alpha+1}$, the spaces $\mathscr{H}_{n_1} $ and $\mathscr{H}_{n_2}$ are orthogonal, since
\begin{align*}
\int_{0}^{2\pi}&C^{\frac{\alpha}{2(\alpha+1)}}_{\frac{n_1}{\alpha+1}}(\cos\varphi) C^{\frac{\alpha+2}{2(\alpha+1)}}_{\frac{n_2-1}{\alpha+1}}(\cos\varphi)  \sin \varphi  d\varphi  =0,
\end{align*}
since the integrand is odd.

Now, let $f \in L^2\left((0,2\pi),\frac{1}{\alpha+1}|\sin \varphi|^{\frac{\alpha}{\alpha +1}} 
  \, d\varphi\right) $.  We want to show that if $f  $  is orthogonal to every $\mathscr{H}_n$ with $n\equiv0,1 \pmod{\alpha+1}$, then $f=0$.  First we decompose $f$ in its even and odd part (with respect to the center of the interval, i.e., $\pi$)
  \[
  f= f_e+ f_o,
  \]
  where
    \[
  f_e(\varphi)= \frac{f(\varphi)+f(2\pi-\varphi)}{2}, \quad   f_o(\varphi)= \frac{f(\varphi)-f(2\pi-\varphi)}{2}.
  \]
First suppose that $f$ is orthogonal to every $\mathscr H_n$ with  $n\equiv 0 \pmod{\alpha+1}$, that is for every $m \in \mathbb{N}_0$ we have 
\begin{align*}
0&=\int_{0}^{2\pi}f(\varphi)  C^{\frac{\alpha}{2(\alpha+1)}}_{m}(\cos\varphi) |\sin \varphi|^{\frac{\alpha}{\alpha +1}}  d\varphi\\
&=\int_{0}^{2\pi}(f_e(\varphi)+f_o(\varphi))  C^{\frac{\alpha}{2(\alpha+1)}}_{m}(\cos\varphi) |\sin \varphi|^{\frac{\alpha}{\alpha +1}} d\varphi\\
&=2\int_{0}^{\pi}f_e(\varphi)C^{\frac{\alpha}{2(\alpha+1)}}_{m}(\cos\varphi) |\sin \varphi|^{\frac{\alpha}{\alpha +1}} d\varphi. 
\end{align*}
This implies that $f_e=0$ by the completeness of the Gegenbauer polynomials $\left\{C^{\frac{\alpha}{2(\alpha+1)}}_{m}\right\}_{m \in \mathbb{N}_0}$ in the space $L^2\left((-1,1),(1-x^2)^{-\frac{1}{2(\alpha+1)}} \right)$. Consequently, $f=f_o$. 

 Then, suppose that $f = f_o$ is orthogonal to every $\mathscr H_n$ with  $n\equiv 1\pmod{\alpha+1}$, that is, for every $m \in \mathbb{N}_0$ we have 
 \begin{align*}
0&=\int_{0}^{2\pi}f_o(\varphi) C^{\frac{\alpha+2}{2(\alpha+1)}}_{m}(\cos\varphi)  \sin \varphi d\varphi\\
&=2\int_{0}^{\pi}f_o(\varphi)C^{\frac{\alpha+2}{2(\alpha+1)}}_{m}(\cos\varphi)  \sin \varphi d\varphi\\
&=2\int_{0}^{\pi}\frac{f_o(\varphi)}{(\sin\varphi)^{\frac1{\alpha+1}}} C^{\frac{\alpha+2}{2(\alpha+1)}}_{m}(\cos\varphi)  (\sin \varphi)^\frac{\alpha+2}{\alpha+1} d\varphi,
\end{align*}
    that implies $\frac{f_o(\varphi)}{(\sin\varphi)^{\frac1{\alpha+1}}} =0 $, by the completeness of Gegenbauer polynomials $\left\{C^{\frac{\alpha+2}{2(\alpha+1)}}_{m}\right\}_{m \in \mathbb{N}_0}$ in the space $L^2\left((-1,1),(1-x^2)^{\frac{1}{2(\alpha+1)}} \right)$. Hence, $f_o=0$. Therefore $f=0$. 
\end{proof}

From the previous Theorem \ref{ortk1} and equation \eqref{eq:eiga1} together with the subsequent computations, we can immediately deduce the following.

\begin{theorem}
Let $k=1$. The values taken by the eigenvalues $\{\mu_j\}_{j \in \mathbb{N}_0}$ of the spherical operator $-\mathcal{L}$ are given by 
\[
\{\mu_j:j \in \mathbb{N}_0 \}= \left\{{\frac{n(n+\alpha)}{(\alpha+1)^2}}: n \in \mathbb{N}_0, n \equiv 0,1 \pmod{\alpha+1}\right\}. 
\]
For each admissible $n$, the corresponding eigenspace is $\mathscr{H}_n$ as in \eqref{autosp_1}, and the eigenvalue is simple. Equivalently, the eigenvalues are given by the two disjoint families
\[
\{\mu_j:j \in \mathbb{N}_0\}=  \left\{i\left(i+\frac{\alpha}{\alpha+1}\right) :i \in \mathbb{N}_0\right\}  \cup \,
\left\{(i+1)\left(i+1-\frac{\alpha}{\alpha+1}\right):i \in \mathbb{N}_0   \right\} .  
\]
\end{theorem}
\begin{proof}
The first set equality and the statement about the eigenspaces and multiplicities follow from the previous computations and Theorem \ref{ortk1}. For what concerns the second set equality, if $n\equiv 0\pmod{\alpha+1}$, then $n=i(\alpha+1)$ for some $i\in \mathbb{N}_0$. Then
\[
\frac{n(n+\alpha)}{(\alpha+1)^2}= \frac{i(\alpha+1)(i(\alpha+1)+\alpha)}{(\alpha+1)^2} = i\left(i+\frac{\alpha}{\alpha+1}\right).
\]
If $n\equiv 1\pmod{\alpha+1}$, then $n=i(\alpha+1)+1$ for some $i\in \mathbb{N}_0$. Then
\begin{align*}
\frac{n(n+\alpha)}{(\alpha+1)^2}&= \frac{(i(\alpha+1)+1)(i(\alpha+1)+1+\alpha)}{(\alpha+1)^2} \\
&= \frac{(i(\alpha+1)+1)(i+1)(\alpha+1)}{(\alpha+1)^2}\\
&=(i+1)\left(i+1-\frac{\alpha}{\alpha+1}\right).
\end{align*}
Finally, we prove that the two families are disjoint. Suppose by contradiction that for some $i_1,i_2 \in \mathbb{N}_0$
\[
i_1\left(i_1+\frac{\alpha}{\alpha+1}\right) = (i_2+1)\left(i_2+1-\frac{\alpha}{\alpha+1}\right).
\]
This yields 
\[
\frac{\alpha}{\alpha+1} = i_2-i_1+1,
\]
which is impossible.
\end{proof}


\section{The case  $k \geq 2$}
Let us now consider the case $k\geq2$. In this case, the polar coordinates are $\rho\in(0,+\infty)$,  $\eta \in \mathbb{S}^{k-1}$, and $\varphi\in\left(-\frac{\pi}{2},\frac{\pi}{2}\right)$. We recall that the spherical part of the  Grushin Laplacian in polar coordinates (see \eqref{GrPo}) is given by
\[
\mathcal{L} =\frac{\partial^2}{\partial\varphi^2}+\left(\frac{\alpha}{\alpha+1}\frac{\cos\varphi}{\sin\varphi}-(k-1)\frac{\sin\varphi}{\cos\varphi}\right)\frac{\partial}{\partial\varphi}+\frac{1}{\cos^2\varphi}\Delta_{\mathbb{S}^{k-1}}.
\]
 As in the previous section, we are searching for solutions to \eqref{eq:eiga}, that is
\begin{equation}\label{lambda_sigma}
-\mathcal{L} g= \frac{n(n+Q-2)}{(\alpha+1)^2}g.
\end{equation}
We make the ansatz that a solution can be written as a product of a function in the variable $\varphi$ and a spherical harmonic on $\mathbb{S}^{k-1}$. Therefore, let $g(\varphi,\eta)=h(\varphi)Y_\ell(\eta)$, where $Y_\ell(\eta)$ is a spherical harmonic of degree $\ell$ on the sphere $\mathbb{S}^{k-1}$. Since $\Delta_\eta Y_\ell=-\ell(\ell+k-2)Y_\ell$, equation \eqref{lambda_sigma} gives
\[
h''(\varphi) +\left(\frac{\alpha}{\alpha+1}\frac{\cos\varphi}{\sin\varphi}-(k-1)\frac{\sin\varphi}{\cos\varphi}\right)h'(\varphi) +\left(\frac{n(n+Q-2)}{(\alpha+1)^2}-\frac{\ell(\ell+k-2)}{\cos^2\varphi}\right)h(\varphi)=0.
\]
Again with the change of variable $\tau=\cos\varphi$ the equation becomes
\[
(1-\tau^2)u''+\tau\left(-1-\frac{\alpha}{\alpha+1}+(k-1)\frac{1-\tau^2}{\tau^2}\right)u'+\left(\frac{n(n+Q-2)}{(\alpha+1)^2}-\frac{\ell(\ell+k-2)}{\tau^2}\right)u=0,
\]
or equivalently
\begin{equation}\label{hyp:eq}
\tau^2(1-\tau^2)u''+\tau\left(\left(1-k-\frac{2\alpha+1}{\alpha+1}\right)\tau^2+k-1\right)u'-\left(\ell(\ell+k-2)-\tau^2\frac{n(n+Q-2)}{(\alpha+1)^2}\right)u=0. 
\end{equation}
It can be checked that a solution to \eqref{hyp:eq} is given by
\[
u(\tau)= \tau^\ell {_2F_1}\left(\frac{\ell(\alpha+1)-n}{2(\alpha+1)}, \frac{\ell(\alpha+1)+n+Q-2}{2(\alpha+1)}; \ell+\frac{k}{2};\tau^2\right), 
\]
where $_2F_1$ is the Gauss hypergeometric function (see, e.g. \cite[\S 2.1]{Er81}). It is well known that if the first parameter of $_2F_1$ is a nonpositive integer, we obtain a Jacobi polynomial (see \cite[\S 4.21]{Sz75}). In our case, if we set
\[
\widetilde{k}^1_{n,\ell}:=\frac{n-\ell(\alpha+1)}{2(\alpha+1)}
\]
and we assume that $\ell$ is such that
\[
\widetilde{k}^1_{n,\ell}\in\mathbb{N}_0,
\]
we obtain
\[
u(\tau)=\tau^\ell {_2F_1}\left(-\widetilde{k}^1_{n,\ell}, p+\widetilde{k}^1_{n,\ell},\gamma, \tau^2\right)=\tau^\ell\frac{\widetilde{k}^1_{n,\ell}!}{(\gamma)_{\widetilde{k}^1_{n,\ell}}}(-1)^{\widetilde{k}^1_{n,\ell}} P_{\widetilde{k}^1_{n,\ell}}^{(p-\gamma,\gamma-1)}(2\tau^2-1), 
\]
where $\gamma=\ell+\frac{k}{2}$, $p=\gamma-\frac{1}{2(\alpha+1)}$, $(\cdot)_{\widetilde{k}^1_{n,\ell}}$ is the Pochhammer symbol and $P_n^{(\alpha,\beta)}$ is a Jacobi polynomial (see \cite[Ch 4]{Sz75}). Therefore
\[
h(\varphi)=\frac{\widetilde{k}^1_{n,\ell}!}{(\gamma)_{\widetilde{k}^1_{n,\ell}}}(-1)^{\widetilde{k}^1_{n,\ell}}(\cos\varphi)^\ell P_{\widetilde{k}^1_{n,\ell}}^{\left(-\frac{1}{2(\alpha+1)},\ell+\frac{k}{2}-1\right)}(\cos2\varphi)
\]
and 
\begin{equation}\label{gphiP}
g(\varphi,\eta)=\frac{\widetilde{k}^1_{n,\ell}!}{(\gamma)_{\widetilde{k}^1_{n,\ell}}}(-1)^{\widetilde{k}^1_{n,\ell}}(\cos\varphi)^\ell P_{\widetilde{k}^1_{n,\ell}}^{\left(-\frac{1}{2(\alpha+1)},\ell+\frac{k}{2}-1\right)}(\cos2\varphi)Y_\ell(\eta).
\end{equation}
Observe that if $\ell$ is even, $\widetilde{k}^1_{n,\ell}=\frac{n-\ell(\alpha+1)}{2(\alpha+1)}\in\mathbb{N}_0$ if and only if  $n\equiv 0\pmod{2\alpha+2}$, while if $\ell$ is odd, if and only if  $n\equiv \alpha+1 \pmod{2\alpha+2}$.   Therefore, we can find an $\ell$ for which we have solutions  of the form \eqref{gphiP}  if and only if $n\equiv 0\pmod {\alpha+1}$. 

Then  we fix $n\equiv 0\pmod {\alpha+1}$, i.e. $n=i(\alpha+1)$ for some $i\in\mathbb{N}_0$, and we denote by  
 \begin{equation}\label{dell}
d_\ell := \binom{\ell+k-1}{\ell}-\binom{\ell+k-3}{\ell-2}
 \end{equation}
the dimension of the space of spherical harmonics of degree $\ell$, where the second binomial coefficient is understood to be zero for $\ell=0,1$. Then the dimension of the space generated by the functions in \eqref{gphiP} when $0\leq \ell \leq i$, $\ell \equiv i \pmod{2}$  considering also the dimension of the spherical harmonics of degree $\ell$ is exactly
\[
\sum_{\substack{\ell=0\\\ell \equiv i\!\!\! \pmod{2}}}^{i}d_\ell = \sum_{j=0}^{\left[\frac{i}{2}\right]}d_{i-2j} =\sum_{j=0}^{\left[\frac{i}{2}\right]}\left(\binom{i-2j+k-1}{i-2j}-\binom{i-2j+k-3}{i-2j-2}\right) = \binom{i+k-1}{i},
\]
where the last equality follows since the sum is telescopic.

In analogy with the previous section, we now try to construct a solution also when $n\equiv 1\pmod{\alpha+1}$. To do so we look for a solution  of the form 
\[
g(\varphi,\eta)=\sin\varphi|\sin\varphi|^{-\frac{\alpha}{\alpha+1}}h(\varphi)Y_\ell(\eta). 
\]
After easy computations, we find that $h$ solves the equation
\[
h''(\varphi) +\left(\frac{\alpha+2}{\alpha+1}\frac{\cos\varphi}{\sin\varphi}-(k-1)\frac{\sin\varphi}{\cos\varphi}\right)h'(\varphi) +\left(\frac{(n-1)(n+Q-1)}{(\alpha+1)^2}-\frac{\ell(\ell+k-2)}{\cos^2\varphi}\right)h(\varphi)=0.
\]
Again with the change of variable $\tau=\cos\varphi$, the equation becomes
\begin{equation}\label{hyp:eq2}
\tau^2(1-\tau^2)u''+\tau\left(\left(1-k-\frac{2\alpha+3}{\alpha+1}\right)\tau^2+k-1\right)u'-\left(\ell(\ell+k-2)-\tau^2\frac{(n-1)(n+Q-1)}{(\alpha+1)^2}\right)u=0. 
\end{equation}
As before, it can be checked that if 
\[
\widetilde{k}^2_{n,\ell}:=\frac{n-1-\ell(\alpha+1)}{2(\alpha+1)}\in\mathbb{N}_0,
\]
a solution to \eqref{hyp:eq2} is given by
\[
u(\tau)=\tau^\ell {_2F_1}\left(-\widetilde{k}^2_{n,\ell}, p+\widetilde{k}^2_{n,\ell},\gamma, \tau^2\right)=\tau^\ell\frac{\widetilde{k}^2_{n,\ell}!}{(\gamma)_{\widetilde{k}^2_{n,\ell}}}(-1)^{\widetilde{k}^2_{n,\ell}} P_{\widetilde{k}^2_{n,\ell}}^{(p-\gamma,\gamma-1)}(2\tau^2-1), 
\]
where  $\gamma=\ell+\frac{k}{2}$, and $p=\gamma+\frac{1}{2(\alpha+1)}$.
 Therefore
\[
h(\varphi)=\frac{\widetilde{k}^2_{n,\ell}!}{(\gamma)_{\widetilde{k}^2_{n,\ell}}}(-1)^{\widetilde{k}^2_{n,\ell}}(\cos\varphi)^\ell P_{\widetilde{k}^2_{n,\ell}}^{\left(\frac{1}{2(\alpha+1)},\ell+\frac{k}{2}-1\right)}(\cos2\varphi)
\]
and 
\begin{equation}\label{gphiP2}
g(\varphi,\eta)=\frac{\widetilde{k}^2_{n,\ell}!}{(\gamma)_{\widetilde{k}^2_{n,\ell}}}(-1)^{\widetilde{k}^2_{n,\ell}}\sin\varphi|\sin\varphi|^{-\frac{\alpha}{\alpha+1}}(\cos\varphi)^\ell P_{\widetilde{k}^2_{n,\ell}}^{\left(\frac{1}{2(\alpha+1)},\ell+\frac{k}{2}-1\right)}(\cos2\varphi)Y_\ell(\eta). 
\end{equation}

Similarly to the previous case, we observe that if $\ell$ is even, then $\widetilde{k}^2_{n,\ell}=\frac{n-1-\ell(\alpha+1)}{2(\alpha+1)}\in\mathbb{N}_0$ if and only if  $n\equiv 1 \pmod{2\alpha+2}$, while if $\ell$ is odd, if and only if $n\equiv \alpha+2\pmod{2\alpha+2}$. Therefore, we can find an $\ell$ for which we have solutions of the form \eqref{gphiP2} if and only if $n\equiv 1\pmod{\alpha+1}$.
We have then constructed two families of solutions, for $n \equiv 0\pmod{\alpha+1}$ and $n \equiv 1 \pmod{\alpha+1}$. Again  note that the dimension of the space generated by the functions in \eqref{gphiP2} for a fixed $n=i(\alpha+1)+1$, $i\in\mathbb{N}_0$, when  $0\leq \ell \leq i$, $\ell \equiv i \pmod{2}$  considering the dimension of the spherical harmonics of degree $\ell$, is exactly
\[
\sum_{\substack{\ell=0\\\ell \equiv i\!\!\! \pmod{2}}}^{i}d_\ell = \sum_{j=0}^{\left[\frac{i}{2}\right]}d_{i-2j} =\sum_{j=0}^{\left[\frac{i}{2}\right]}\left(\binom{i-2j+k-1}{i-2j}-\binom{i-2j+k-3}{i-2j-2}\right) = \binom{i+k-1}{i}.
\]

 Now fix $\ell \in \mathbb{N}_0$ and denote by $\{Y_{\ell,j}\}_{j=1,\ldots,d_\ell}$ an orthonormal  basis of spherical harmonics of degree $\ell$ in $\mathbb{S}^{k-1}$. We set 
\begin{equation}\label{autosp_2}
\mathscr{H}_n := \begin{cases}
&\mathrm{span} \Bigg\{\frac{\widetilde{k}^1_{n,\ell}!}{(\gamma)_{\widetilde{k}^1_{n,\ell}}}(-1)^{\widetilde{k}^1_{n,\ell}}(\cos\varphi)^\ell P_{\widetilde{k}^1_{n,\ell}}^{\left(-\frac{1}{2(\alpha+1)},\ell+\frac{k}{2}-1\right)}(\cos2\varphi)Y_{\ell,j}(\eta)\\
&\hspace{2cm}:  \ell\in \mathbb{N}_0 \mbox{ s.t. } \widetilde{k}^1_{n,\ell} \in \mathbb{N}_0,\, j=1,\ldots,d_\ell \Bigg\} \qquad
\mbox{ if } n \equiv 0\pmod{\alpha+1}\\
&\mathrm{span} \Bigg\{ \frac{\widetilde{k}^2_{n,\ell}!}{(\gamma)_{\widetilde{k}^2_{n,\ell}}}(-1)^{\widetilde{k}^2_{n,\ell}}\sin\varphi|\sin\varphi|^{-\frac{\alpha}{\alpha+1}}(\cos\varphi)^\ell P_{\widetilde{k}^2_{n,\ell}}^{\left(\frac{1}{2(\alpha+1)},\ell+\frac{k}{2}-1\right)}(\cos2\varphi)Y_{\ell,j}(\eta)\\
&\hspace{2cm}: \ell\in \mathbb{N}_0 \mbox{ s.t. } \widetilde{k}^2_{n,\ell} \in \mathbb{N}_0,\, j=1,\ldots,d_\ell \Bigg\} \qquad 
\mbox{ if } n \equiv 1\pmod{\alpha+1}.
\end{cases}
\end{equation}
By the above computation, for $n \equiv 0,1\pmod{\alpha+1}$ we have 
\begin{equation}\label{dim}
\dim \mathscr{H}_n = \begin{cases}
& \dbinom{\frac{n}{\alpha+1}+k-1}{\frac{n}{\alpha+1}} \qquad\qquad
\mbox{ if } n \equiv 0\pmod{\alpha+1}\vspace{.3cm}\\ 
&\dbinom{\frac{n-1}{\alpha+1}+k-1}{\frac{n-1}{\alpha+1}} \qquad\qquad 
\mbox{ if } n \equiv 1 \pmod{\alpha+1}.
\end{cases}
\end{equation}
In the case $k\geq2$, according to the parametrization of $\SA$, we have the identification
\[
L_\alpha^2(\SA) = L^2\left(\left(-\frac\pi2,\frac\pi2\right)\times \mathbb{S}^{k-1},\frac{1}{(\alpha+1)^k}|\sin \varphi|^{\frac{\alpha}{\alpha +1}} |\cos\varphi|^{k-1} \, d\varphi d\mathcal{H}^{k-1}\right).
\]
Indeed, it is easily seen that  on $\SA$
\[
 d\mathcal{H}_\alpha^{N-1} = \frac{1}{(\alpha+1)^k}|\sin \varphi|^{-\frac{\alpha}{\alpha+1}} |\cos\varphi|^{k-1} 
  \, d\varphi d\mathcal{H}^{k-1},\qquad  \psi_\alpha = |\sin \varphi|^{\frac{2\alpha}{\alpha+1}}.
 \]

We have the following theorem.
\begin{theorem}\label{ortk2}
Let $k\geq2$ and let $\mathscr{H}_n $ be as in \eqref{autosp_2}. The following orthogonal decomposition holds:
\[
L^2\left(\left(-\frac\pi2,\frac\pi2\right)\times \mathbb{S}^{k-1},\frac{1}{(\alpha+1)^k}|\sin \varphi|^{\frac{\alpha}{\alpha +1}} |\cos\varphi|^{k-1} \, d\varphi d\mathcal{H}^{k-1}\right) = \bigoplus_{\substack{n \in \mathbb{N}_0\\ n\equiv 0,1 \!\!\!\!\pmod{\alpha+1}}}\mathscr{H}_n.
\]
\end{theorem}
\begin{proof}
We first prove that $\mathscr{H}_n $ are mutually orthogonal. Indeed, let us first see that
if $n_1 \equiv 0$, $n_2 \equiv 0\pmod{\alpha+1}$, $n_1 \neq n_2$, then the spaces $\mathscr{H}_{n_1} $ and $\mathscr{H}_{n_2}$ are orthogonal.  By the orthogonality of spherical harmonics, it suffices to show that 
\[
\int_{-\frac\pi2}^{\frac\pi2}(\cos\varphi)^{2\ell} P_{\widetilde{k}^1_{n_1,\ell}}^{\left(-\frac{1}{2(\alpha+1)},\ell+\frac{k}{2}-1\right)}(\cos2\varphi)P_{\widetilde{k}^1_{n_2,\ell}}^{\left(-\frac{1}{2(\alpha+1)},\ell+\frac{k}{2}-1\right)}(\cos2\varphi)|\sin \varphi|^{\frac{\alpha}{\alpha +1}}(\cos\varphi)^{k-1}d\varphi=0.
\]
By the parity of the integrand, we have 
\begin{align*}
\int_{-\frac\pi2}^{\frac\pi2}&(\cos\varphi)^{2\ell} P_{\widetilde{k}^1_{n_1,\ell}}^{\left(-\frac{1}{2(\alpha+1)},\ell+\frac{k}{2}-1\right)}(\cos2\varphi)P_{\widetilde{k}^1_{n_2,\ell}}^{\left(-\frac{1}{2(\alpha+1)},\ell+\frac{k}{2}-1\right)}(\cos2\varphi)|\sin \varphi|^{\frac{\alpha}{\alpha +1}}(\cos\varphi)^{k-1}d\varphi\\
&=2\int_{0}^{\frac\pi2}(\cos\varphi)^{2\ell} P_{\widetilde{k}^1_{n_1,\ell}}^{\left(-\frac{1}{2(\alpha+1)},\ell+\frac{k}{2}-1\right)}(\cos2\varphi)P_{\widetilde{k}^1_{n_2,\ell}}^{\left(-\frac{1}{2(\alpha+1)},\ell+\frac{k}{2}-1\right)}(\cos2\varphi)|\sin \varphi|^{\frac{\alpha}{\alpha +1}}(\cos\varphi)^{k-1}d\varphi.
\end{align*}
Now we use the change of variable 
$
\cos(2\varphi)=x
$
for which 
$
dx=-2\sin(2\varphi)d\varphi=-4\sin\varphi\cos\varphi d\varphi,
$
and
\[
\cos^2(\varphi)=\frac{1+x}{2},\quad \sin^2(\varphi)=\frac{1-x}{2}.
\]
Accordingly, 
\[
d\varphi = -\frac{1}{2\sqrt{1-x}\sqrt{1+x}}dx.
\]
Then we have
\begin{align*}
&\int_{0}^{\frac\pi2}(\cos\varphi)^{2\ell} P_{\widetilde{k}^1_{n_1,\ell}}^{\left(-\frac{1}{2(\alpha+1)},\ell+\frac{k}{2}-1\right)}(\cos2\varphi)P_{\widetilde{k}^1_{n_2,\ell}}^{\left(-\frac{1}{2(\alpha+1)},\ell+\frac{k}{2}-1\right)}(\cos2\varphi)|\sin \varphi|^{\frac{\alpha}{\alpha +1}}(\cos\varphi)^{k-1}d\varphi\\
&=\int_{-1}^{1}\left(\frac{1+x}{2}\right)^{\ell+\frac{k-1}{2}}\left(\frac{1-x}{2}\right)^{\frac{\alpha}{2(\alpha+1)}}P_{\widetilde{k}^1_{n_1,\ell}}^{\left(-\frac{1}{2(\alpha+1)},\ell+\frac{k}{2}-1\right)}(x)P_{\widetilde{k}^1_{n_2,\ell}}^{\left(-\frac{1}{2(\alpha+1)},\ell+\frac{k}{2}-1\right)}(x)\frac{1}{2\sqrt{1-x}\sqrt{1+x}}dx\\
&=\frac{1}{2^{\ell+\frac{k+1}{2}+\frac{\alpha}{2(\alpha+1)}}}\int_{-1}^{1}\left(1+x\right)^{\ell+\frac{k}{2}-1}\left(1-x\right)^{-\frac{1}{2(\alpha+1)}}P_{\widetilde{k}^1_{n_1,\ell}}^{\left(-\frac{1}{2(\alpha+1)},\ell+\frac{k}{2}-1\right)}(x)P_{\widetilde{k}^1_{n_2,\ell}}^{\left(-\frac{1}{2(\alpha+1)},\ell+\frac{k}{2}-1\right)}(x) dx=0,
\end{align*}
by the orthogonality of Jacobi polynomials. 

As a second case, we prove that if $n_1 \equiv  1$, $n_2 \equiv 1 \pmod{\alpha+1}$, $n_1 \neq n_2$, then the spaces $\mathscr{H}_{n_1} $ and $\mathscr{H}_{n_2}$ are orthogonal.  By the orthogonality of spherical harmonics it suffices to show that 
\[
\int_{-\frac\pi2}^{\frac\pi2}(\cos\varphi)^{2\ell} P_{\widetilde{k}^2_{n_1,\ell}}^{\left(\frac{1}{2(\alpha+1)},\ell+\frac{k}{2}-1\right)}(\cos2\varphi)P_{\widetilde{k}^2_{n_2,\ell}}^{\left(\frac{1}{2(\alpha+1)},\ell+\frac{k}{2}-1\right)}(\cos2\varphi)|\sin \varphi|^{2-\frac{\alpha}{\alpha +1}}(\cos\varphi)^{k-1}d\varphi=0.
\]
By the parity of the integrand we have that
\begin{align*}
&\int_{-\frac\pi2}^{\frac\pi2}(\cos\varphi)^{2\ell} P_{\widetilde{k}^2_{n_1,\ell}}^{\left(\frac{1}{2(\alpha+1)},\ell+\frac{k}{2}-1\right)}(\cos2\varphi)P_{\widetilde{k}^2_{n_2,\ell}}^{\left(\frac{1}{2(\alpha+1)},\ell+\frac{k}{2}-1\right)}(\cos2\varphi)|\sin \varphi|^{2-\frac{\alpha}{\alpha +1}}(\cos\varphi)^{k-1}d\varphi\\
&=2\int_{0}^{\frac\pi2}(\cos\varphi)^{2\ell} P_{\widetilde{k}^2_{n_1,\ell}}^{\left(\frac{1}{2(\alpha+1)},\ell+\frac{k}{2}-1\right)}(\cos2\varphi)P_{\widetilde{k}^2_{n_2,\ell}}^{\left(\frac{1}{2(\alpha+1)},\ell+\frac{k}{2}-1\right)}(\cos2\varphi)|\sin \varphi|^{2-\frac{\alpha}{\alpha +1}}(\cos\varphi)^{k-1}d\varphi.
\end{align*}
With the same change of variable as above we get
\begin{align*}
&\int_{0}^{\frac\pi2}(\cos\varphi)^{2\ell} P_{\widetilde{k}^2_{n_1,\ell}}^{\left(\frac{1}{2(\alpha+1)},\ell+\frac{k}{2}-1\right)}(\cos2\varphi)P_{\widetilde{k}^2_{n_2,\ell}}^{\left(\frac{1}{2(\alpha+1)},\ell+\frac{k}{2}-1\right)}(\cos2\varphi)|\sin \varphi|^{2-\frac{\alpha}{\alpha +1}}(\cos\varphi)^{k-1}d\varphi\\
&=\int_{-1}^{1}\left(\frac{1+x}{2}\right)^{\ell+\frac{k-1}{2}}\left(\frac{1-x}{2}\right)^{1-\frac{\alpha}{2(\alpha+1)}}P_{\widetilde{k}^2_{n_1,\ell}}^{\left(\frac{1}{2(\alpha+1)},\ell+\frac{k}{2}-1\right)}(x)P_{\widetilde{k}^2_{n_2,\ell}}^{\left(\frac{1}{2(\alpha+1)},\ell+\frac{k}{2}-1\right)}(x)\frac{1}{2\sqrt{1-x}\sqrt{1+x}}dx\\
&=\frac{1}{2^{\ell+\frac{k+3}{2}-\frac{\alpha}{2(\alpha+1)}}}\int_{-1}^{1}\left(1+x\right)^{\ell+\frac{k}{2}-1}\left(1-x\right)^{\frac{1}{2(\alpha+1)}}P_{\widetilde{k}^2_{n_1,\ell}}^{\left(\frac{1}{2(\alpha+1)},\ell+\frac{k}{2}-1\right)}(x)P_{\widetilde{k}^2_{n_2,\ell}}^{\left(\frac{1}{2(\alpha+1)},\ell+\frac{k}{2}-1\right)}(x) dx=0,
\end{align*}
again by the orthogonality of Jacobi polynomials.

Finally, we prove that if $n_1 \equiv  0$, $n_2 \equiv 1 \pmod{\alpha+1}$, then the spaces $\mathscr{H}_{n_1} $ and $\mathscr{H}_{n_2}$ are orthogonal.  By the orthogonality of spherical harmonics, it suffices to show that 
\[
\int_{-\frac\pi2}^{\frac\pi2}(\cos\varphi)^{2\ell} P_{\widetilde{k}^1_{n_1,\ell}}^{\left(-\frac{1}{2(\alpha+1)},\ell+\frac{k}{2}-1\right)}(\cos2\varphi)P_{\widetilde{k}^2_{n_2,\ell}}^{\left(\frac{1}{2(\alpha+1)},\ell+\frac{k}{2}-1\right)}(\cos2\varphi)\sin \varphi (\cos\varphi)^{k-1}d\varphi=0,
\]
that holds true since the integrand is odd.

Now let $f \in L^2\left(\left(-\frac\pi2,\frac\pi2\right)\times \mathbb{S}^{k-1},\frac{1}{(\alpha+1)^k}|\sin \varphi|^{\frac{\alpha}{\alpha +1}} |\cos\varphi|^{k-1} \, d\varphi d\mathcal{H}^{k-1}\right) $. We want to show that if $f  $  is orthogonal to every $\mathscr H_n$, then $f=0$. First we decompose $f$ in its even and odd part with respect to $\varphi \in (-\frac{\pi}{2},\frac{\pi}{2})$, i.e., 
  \[
  f= f_e+ f_o,
  \]
  where
    \[
  f_e(\eta,\varphi)= \frac{f(\eta,\varphi)+f(\eta,-\varphi)}{2}, \quad   f_o(\eta,\varphi)= \frac{f(\eta,\varphi)-f(\eta,-\varphi)}{2}.
  \]
First suppose that $f$ is orthogonal to every $\mathscr H_n$, with  $n\equiv 0\pmod{\alpha+1}$. Notice that, since $n=i(\alpha+1)$ for some $i\in\mathbb{N}_0$, setting 
\[
m=\widetilde{k}^1_{n,\ell}=\frac{i(\alpha+1)-\ell(\alpha+1)}{2(\alpha+1)}=\frac{i-\ell}{2},
\]
implies that, for a fixed $\ell$, the number $m$ can range over all nonnegative integers, by appropriately choosing $i$. The above formula also shows that after fixing $m$, taking $n=(\ell+2m)(\alpha+1)$ allows $\ell$ to vary through all nonnegative integers. 
Therefore, for every $m \in \mathbb{N}_0$, $\ell \in\mathbb{N}_0$ and every 
$j = 1,\ldots, d_\ell$, we have 
\begin{align*}
\int_{-\frac\pi2}^{\frac\pi2}\int_{\mathbb{S}^{k-1}}f(\eta,\varphi) (\cos\varphi)^\ell P_{m}^{\left(-\frac{1}{2(\alpha+1)},\ell+\frac{k}{2}-1\right)}(\cos2\varphi)Y_{\ell,j}(\eta)     
|\sin \varphi|^{\frac{\alpha}{\alpha +1}}(\cos\varphi)^{k-1} 
  \,  d\mathcal{H}^{k-1}d\varphi=0.
\end{align*}
Note that, since $(\cos\varphi)^\ell P_{m}^{\left(-\frac{1}{2(\alpha+1)},\ell+\frac{k}{2}-1\right)}(\cos2\varphi)Y_{\ell,j}(\eta)   
|\sin \varphi|^{\frac{\alpha}{\alpha +1}}(\cos\varphi)^{k-1}$ is even in $\varphi$, the previous equation implies  
\begin{align*}
\int_{-\frac\pi2}^{\frac\pi2}\int_{\mathbb{S}^{k-1}}f_e(\eta,\varphi) (\cos\varphi)^\ell P_{m}^{\left(-\frac{1}{2(\alpha+1)},\ell+\frac{k}{2}-1\right)}(\cos2\varphi)Y_{\ell,j}(\eta)     
|\sin \varphi|^{\frac{\alpha}{\alpha +1}}(\cos\varphi)^{k-1} 
  \,  d\mathcal{H}^{k-1}d\varphi=0.
\end{align*}
By Fubini's Theorem and again by parity, we have 
\begin{align*}
0&=\int_{-\frac\pi2}^{\frac\pi2} (\cos\varphi)^{2\ell} P_{m}^{\left(-\frac{1}{2(\alpha+1)},\ell+\frac{k}{2}-1\right)}(\cos2\varphi)F^e_{\ell,j}(\varphi)     
|\sin \varphi|^{\frac{\alpha}{\alpha +1}}(\cos\varphi)^{k-1} 
  \, d\varphi\\
  & =2\int_{0}^{\frac\pi2} (\cos\varphi)^{2\ell} P_{m}^{\left(-\frac{1}{2(\alpha+1)},\ell+\frac{k}{2}-1\right)}(\cos2\varphi)F^e_{\ell,j}(\varphi)     
|\sin \varphi|^{\frac{\alpha}{\alpha +1}}(\cos\varphi)^{k-1} 
  \, d\varphi, 
\end{align*}
where 
\[
F^e_{\ell,j}(\varphi):= \cos(\varphi)^{-\ell}\int_{\mathbb{S}^{k-1}}Y_{\ell,j}(\eta)f_e(\eta,\varphi)\,d\mathcal{H}^{k-1} \qquad \varphi \in \left(-\frac\pi2,\frac\pi2\right).
\]
With the usual change of variable and by the completeness of $\left\{P^{(-\frac{1}{2(\alpha+1)},\ell+\frac{k}{2}-1)}_m\right\}_{m\in\N_0}$  in the space $L^2\left([-1,1], (1-x)^{-\frac{1}{2(\alpha+1)}}(1+x)^{\ell+\frac{k}{2}-1} dx\right)$, we get that 
\[
F^e_{\ell,j}(\varphi) =0, \quad \mbox{ for a.e. } \varphi\in \left(-\frac\pi2,\frac\pi2\right).
\]
Therefore, by the completeness of the spherical harmonics $\{Y_{\ell,j}\}$ we get that $f_e =0$.

Now, suppose that $f=f_o$ is  orthogonal to every $\mathscr H_n$, with  $n\equiv 1\pmod{\alpha+1}$. Arguing as above, this implies that for every $m \in \mathbb{N}_0$, $\ell \in\mathbb{N}_0$ and every 
$j = 1,\ldots, d_\ell$, we have 
\begin{align*}
\int_{-\frac\pi2}^{\frac\pi2}\int_{\mathbb{S}^{k-1}}f_o(\eta,\varphi)\sin\varphi|\sin\varphi|^{-\frac{\alpha}{\alpha+1}} (\cos\varphi)^\ell &P_{m}^{\left(\frac{1}{2(\alpha+1)},\ell+\frac{k}{2}-1\right)}(\cos2\varphi)Y_{\ell,j}(\eta)  \\  
&\times|\sin \varphi|^{\frac{\alpha}{\alpha +1}}(\cos\varphi)^{k-1} 
  \,  d\mathcal{H}^{k-1}d\varphi=0.
\end{align*}
By Fubini's Theorem and by parity, we have 
\begin{align*}
0&=\int_{-\frac\pi2}^{\frac\pi2} (\cos\varphi)^{2\ell} P_{m}^{\left(\frac{1}{2(\alpha+1)},\ell+\frac{k}{2}-1\right)}(\cos2\varphi)F^o_{\ell,j}(\varphi)     
|\sin \varphi|^{2-\frac{\alpha}{\alpha +1}}(\cos\varphi)^{k-1} 
  \, d\varphi \\
  & =2\int_{0}^{\frac\pi2}  (\cos\varphi)^{2\ell} P_{m}^{\left(\frac{1}{2(\alpha+1)},\ell+\frac{k}{2}-1\right)}(\cos2\varphi)F^o_{\ell,j}(\varphi)     
|\sin \varphi|^{2-\frac{\alpha}{\alpha +1}}(\cos\varphi)^{k-1} 
  \, d\varphi, 
\end{align*}
where 
\[
F^o_{\ell,j}(\varphi):= (\sin\varphi)^{-1} |\sin \varphi|^{\frac{\alpha}{\alpha +1}}\cos(\varphi)^{-\ell}\int_{\mathbb{S}^{k-1}}Y_{\ell,j}(\eta)f_o(\eta,\varphi)\,d\mathcal{H}^{k-1}, \qquad \varphi \in\left(-\frac\pi2,\frac\pi2\right).
\]
With the same change of variable as above, and by the completeness of the Jacobi polynomials $\left\{P^{\left(\frac{1}{2(\alpha+1)},\ell+\frac{k}{2}-1\right)}_m\right\}_{m\in\N_0}$  in the space $L^2\left([-1,1], (1-x)^{\frac{1}{2(\alpha+1)}}(1+x)^{\ell+\frac{k}{2}-1} dx\right)$, we get
\[
F^o_{\ell,j}(\varphi) =0 \quad \mbox{ for a.e. } \varphi\in \left(-\frac\pi2,\frac\pi2\right).
\]
Again, by the completeness of the spherical harmonics $\{Y_{\ell,j}\}$, we get that $f_o =0$. Therefore $f=0$.
\end{proof}

From the previous Theorem \ref{ortk2} and equation \eqref{lambda_sigma}, together with the subsequent computations, we can immediately deduce the following theorem. 

\begin{theorem}
Let $k\geq2$. The values taken by the eigenvalues $\{\mu_j\}_{j \in \mathbb{N}_0}$ of the spherical operator $-\mathcal{L}$ are given by 
\[
\{\mu_j:j \in \mathbb{N}_0\}  = \left\{ \frac{n(n+Q-2)}{(\alpha+1)^2}: n \in \mathbb{N}_0, n \equiv 0,1 \pmod{\alpha+1}\right\}. 
\]
For each admissible $n$, the corresponding eigenspace is $\mathscr{H}_n$ as in \eqref{autosp_2}, and the multiplicity of the eigenvalue is 
\begin{equation}
\begin{cases}
& \dbinom{\frac{n}{\alpha+1}+k-1}{\frac{n}{\alpha+1}} \qquad\qquad
\mbox{ if } n \equiv 0\pmod{\alpha+1}\vspace{.3cm}\\ 
&\dbinom{\frac{n-1}{\alpha+1}+k-1}{\frac{n-1}{\alpha+1}} \qquad\qquad 
\mbox{ if } n \equiv 1 \pmod{\alpha+1}.
\end{cases}
\end{equation}
Equivalently, the eigenvalues are given by the two disjoint families
\[
\{\mu_j:j \in \mathbb{N}_0\}=  \left\{i\left(i+\frac{Q-2}{\alpha+1}\right) :i \in \mathbb{N}_0\right\}   \cup \,
\left\{(i+k)\left(i+k-\frac{Q-2}{\alpha+1}\right): i \in \mathbb{N}_0  \right\}.  
\]
\end{theorem}
\begin{proof}
The first set equality and the statement about the multiplicity follow from the previous computations, Theorem \ref{ortk2} and \eqref{dim}. As for the second set equality, if $n\equiv 0\pmod{\alpha+1}$, then $n=i(\alpha+1)$ for some $i\in \mathbb{N}_0$. Then
\[
\frac{n(n+Q-2)}{(\alpha+1)^2}= \frac{i(\alpha+1)(i(\alpha+1)+Q-2)}{(\alpha+1)^2} = i\left(i+\frac{Q-2}{\alpha+1}\right).
\]
If $n\equiv 1\pmod{\alpha+1}$, then $n=i(\alpha+1)+1$ for some $i\in \mathbb{N}_0$. Then
\begin{align*}
\frac{n(n+Q-2)}{(\alpha+1)^2}&= \frac{(i(\alpha+1)+1)(i(\alpha+1)+1+Q-2)}{(\alpha+1)^2} \\
&= \frac{(i(\alpha+1)+1)(i(\alpha+1)+k(\alpha+1))}{(\alpha+1)^2} \\
&= (i+k)\left(i+\frac{1}{\alpha+1}\right)\\
&=(i+k)\left(i+k-\frac{Q-2}{\alpha+1}\right).
\end{align*}
Finally, we prove that these two families are disjoint. Suppose by contradiction that for some $i_1,i_2 \in \mathbb{N}_0$
\[
i_1\left(i_1+\frac{Q-2}{\alpha+1}\right) = (i_2+k)\left(i_2+k-\frac{Q-2}{\alpha+1}\right). 
\]
This yields
\[
\frac{Q-2}{\alpha+1} = i_2-i_1+k
\]
which is impossible since 
\[
\frac{Q-2}{\alpha+1}= k-\frac{1}{\alpha+1}.
\] 
\end{proof}

\subsection*{AI statement}
The authors declare that AI tools were used solely to audit mathematical results and improve the exposition.

\subsection*{Acknowledgments}
The authors are members of the ``Gruppo Nazionale per l’Analisi Matematica, la Probabilità
e le loro Applicazioni'' (GNAMPA) of the “Istituto Nazionale di Alta Matematica” (INdAM). P.L. acknowledges the support of the project GNAMPA
PROGETTI DI RICERCA 2026, CUP E53C25002010001.

\footnotesize
\bibliographystyle{plain}
\bibliography{biblio}
\end{document}